\documentclass[12pt,a4paper]{article}
\usepackage[OT1]{fontenc}
\usepackage[english]{babel}
\usepackage{amsfonts}
\usepackage{amsthm}
\usepackage{amsmath}

\def\h{ {\cal H} }

\def\b{ {\cal B} }
\def\u{ {\cal U} }
\def\v{ {\cal V} }

\def\t{ {\cal T} }
\def\s{ {\cal S} }
\def\e{ {\cal E} }

\def\j{ {\cal J} }
\def\o{ {\cal O} }

\newtheorem{teo}{Theorem}[section]
\newtheorem{prop}[teo]{Proposition}
\newtheorem{lem}[teo]{Lemma}
\newtheorem{coro}[teo]{Corollary}

\theoremstyle{definition}
\newtheorem{rem}[teo]{Remark}

\title{Reflections in $L^2(\mathbb{T})$}
\author{Esteban Andruchow\footnote{{\sc  {Instituto Argentino de Matem\'atica, `Alberto P. Calder\'on', CONICET, Saavedra 15 3er. piso,
(1083) Buenos Aires, Argentina }} and {{\sc Universidad Nacional de General Sarmiento, J.M. Gutierrez 1150, (1613) Los Polvorines, Argentina}} e-mail: eanddruch@campus.ungs.edu.ar}
}

\begin{document}

\maketitle 

\begin{abstract}
Let $\mathbb{D}=\{z\in\mathbb{C}: |z|<1\}$ and $\mathbb{T}=\{z\in\mathbb{C}: |z|=1\}$.
For $a\in\mathbb{D}$, consider $\varphi_a(z)=\frac{a-z}{1-\bar{a}z}$ and $C_a$  the composition operator in $L^2(\mathbb{T})$ induced by $\varphi_a$:
$$
C_a f=f\circ\varphi_a.
$$
Clearly $C_a$ satisties $C_a^2=I$, i.e., is a non-selfadjoint reflection. 
We also consider the following symmetries (selfadjoint reflections) related to $C_a$:
$$
R_a=M_{\frac{|k_a|}{\|k_a\|_2}}C_a \ \hbox{ and } \ W_a=M_{\frac{k_a}{\|k_a\|_2}}C_a,
$$
where $k_a(z)=\frac{1}{1-\bar{a}z}$ is the Szego kernel. The symmetry $R_a$ is the unitary part in the polar decomposition of $C_a$. We characterize the eigenspaces  $N(T_a\pm I)$ for $T_a=C_a, R_a$ or $W_a$, and study their relative positions when one changes the parameter $a$, e.g., $N(T_a\pm I)\cap N(T_b\pm I)$, $N(T_a\pm I)\cap N(T_b\pm I)^\perp$, $N(T_a\pm I)^\perp\cap N(T_b\pm I)$, etc., for $a\ne b\in\mathbb{D}$.
\end{abstract}

\bigskip



\section{Introduction}
Let $\h=L^2(\mathbb{T})$ with normalized Lebesgue measure. For each $a\in\mathbb{D}$ let 
$$
\varphi_a(z):=\displaystyle{\frac{a-z}{1-\bar{a} z}}.
$$
Note that $\varphi_a\circ\varphi_a(z)=z$ and $|\varphi_a(z)|=1$ for $z\in\mathbb{T}$, and therefore the composition operator
$$
C_a:\h\to\h,  \ C_af=f\circ\varphi_a
$$
is well defined, and is a {\it reflection} in $\h$: $C_a^2=I$. A reflection $C$ is a {\it symmetry } if additionally $C^*=C$. Clearly $C_a$ is a symmetry only if $a=0$. The adjoint of $C_a$ is easily computed:
$C_a^*=(1-|a|^2)M_{\frac{1}{|1-\bar{a}z|^2}} C_a$, where $M_g$ denotes the multiplication operator (by $g$).

There are two symmetries closely relaed to $C_a$. First the unitary part $R_a$ in the polar decomposition $C_a=R_a(C_a^*C_a)^{1/2}$ of $C_a$ (in \cite{cpr proyecciones}, G. Corach, H. Porta and L. Recht noted this fact, that the unitary part of a reflection is in fact a symmetry). Next the operator $W_a$ given by $W_a=M_{\psi_a}C_a$, where $\psi_a$ is the  normalized Szego kernel $\psi_a=\frac{(1-|a|^2)^{1/2}}{1-\bar{a}z}$.

The object of this paper is the study of eigenspaces of these symmetries, and their relative position when one changes the parameter $a\in\mathbb{D}$. By relative position of two subspaces $\s$, $\t$, we mean the computation of the intersections $\s\cap\t$, $\s^\perp\cap\t^\perp$, $\s\cap\t^\perp$ and $\s^\perp\cap\t$. Of particular interest to us are the last two, for their geometric significance: if these two intersections have the same dimention, it means that there exists a geodesic of the Grassmann manifold of $\h$ that joins them; if they are trivial, the geodesic is unique. When the four intersections are trivial, the subspaces $\s$, $\t$ are said to be in generic position.

This paper is a continuation of \cite{disco y composicion}, where composition operators with symbols $\varphi_a$ where considered in the Hardy space of the disc. Several results otained there will be useful here.

The contents of the paper are the following. In Section 2 the eigenspaces of $C_a$, $R_a$ and $W_a$ are characterized. In this characterization, an important role is played by the unique fixed point $\omega_a$ of $\varphi_a$ inside $\mathbb{D}$. In Section 3 we make some brief remarks on the maps $a\mapsto C_a$ and $a\mapsto R_a$. In Section 4 we obtain partial results on the relative position of the eigenspaces of these reflections induced by different $a,b\in\mathbb{D}$. The main result is Theorem 4.6, which states that for $a\ne b$
$$
 N(C_a-I)\cap N(C_b-I)=\langle 1 \rangle \ \hbox{ and } \ N(C_a+I)\cap N(C_b+I)=\{0\}.
$$
In Section 5 we consider the special case of the subspaces $N(C_a-I)$ and $N(C_a+I)$, where a complete study can be done. In Section 6, using the conjugation properties
$$
R_aC_aR_a=C_a^*, \ R_aC_0R_a=R_{\Omega_a}, \ W_aC_aW_a=M_{\frac{k_a^2}{\|k_a\|_2^2}}C_a=M_{\frac{k_a}{|k_a|}}C_a^*\  \hbox{ and }\ W_aC_0W_a=W_{\Omega_a}
$$
we can compute further intersections between the eigenspaces. Here $\Omega_a=\frac{2a}{1+|a|^2}$ is characterized by $\omega_{\Omega_a}=a$, i.e., $a$ is the unique fixed point of $\varphi_{\Omega_a}$ in $\mathbb{D}$.
In Section 7 we examine the consequence of these intersections for the geometry of the Grassmann manifold of $\h$: they determine wether given pairs of subspaces can be joined by a geodesic of this manifold. 

{\bf Aknowledgement:} this work was completed with the support of PICT 2019 04060 (ANPCyT), Argentina.

\section{Preliminaries and notations}

If $\s$ is a closed subspace of $\h$, $P_\s$ will denote the orthogonal projection onto $\s$. For a bounded operator $T$, $|T|$ will denote the modulus of $T$, i.e., $|T|=(T^*T)^{1/2}$; $N(T)$ and $R(T)$ will denote the nullspace and range of $T$, respectively.

Note that
\begin{equation} \label{modulo Ca}
C_a^*C_a=(1-|a|^2) M_{\frac{1}{|1-\bar{a}z|^2}} \ \hbox{ and }\  |C_a|=(C_a^*C_a)^{1/2}=(1-|a|^2)^{1/2} M_{\frac{1}{|1-\bar{a}z|}}.
\end{equation}

In \cite{cpr proyecciones}, Corach, Porta and Recht, made the following remarkable observation:
\begin{rem}
Let $C$ be a reflection in $\h$, and let $C=R|C|$ be its polar decomposition. Then the unitary operator $R$ is in fact a symmetry.
\end{rem}
Let us denote by $R_a$ the symmetry obtained in the polar decomposition of $C_a$. After routine computations one obtains that
\begin{equation}\label{Ra}
R_a=C_a |C_a|^{-1}=\frac{1}{(1-|a|^2)^{1/2}}C_a M_{|1-\bar{a}z|}=(1-|a|^2)^{1/2} M_{\frac{1}{|1-\bar{a}z|}} C_a.
\end{equation}
Note that if $C^2=I$ and $C=R|C|$ is the polar decomposition, then $C^*=|C|R$, so that
$RCR=C^*$. In particular,  for $C=C_a$ we have
\begin{equation}\label{boludez}
R_aC_aR_a=C_a^*.
\end{equation}

There is another symmetry in $\h$ associated with $\varphi_a$, known in the literature (at least when acting in the (invariant) Hardy subspace $H^2(\mathbb{T})\subset \h$,  \cite{cowenmccluer}):
\begin{equation}\label{Wa}
W_a=M_{\psi_a}C_a=(1-|a|^2)^{1/2} M_{\frac{1}{1-\bar{a}z}} C_a.
\end{equation}

Note the similarity between $R_a$ in (\ref{Ra}) and $W_a$ in (\ref{Wa}).

We shall say that an element $f\in\h$ is {\it even} if $f(z)=f(-z) \ p.p.$ in $\mathbb{T}$. Equivalently, if the Fourier coefficients $\hat{f}(n)$ of $f$ are $0$ for $n$ odd. We shall denote by $\e$ the closed subspace of even elements. Similarly, $\o$ will denote the subspace of {\it odd} elements of $\h$ ($g\in\h$ such that $g(-z)=-g(z)  \ p.p.$ in $\mathbb{T}$ or $\hat{g}(n)=0$ for $n$ even). Note that for $a=0$, $C_0=R_0$, and $N(C_0-I)=\e$ and $N(C_0+I)=\o$.

Finally, let $\h^+$ be the Hardy space, as a subset of $\h$, $\h^+:=H^2(\mathbb{T})=\{f\in\h: \hat{f}(n)=0 \hbox{ for } n<0\}$, identified with the Hardy space of the disk $H^2(\mathbb{D})=\{g(z)=\sum_{k=0}^\infty a_k z^k: \sum_{k=0}^\infty |a_k|^2<\infty\}.$

A reflection $T$ is completely described by its two eigenspaces $N(T-I)$ and $N(T+1)$. 
The object of this study is the characterization of the eigenspaces $N(T-I)$ and $N(T+I)$ for $T=C_a$,  $R_a$ or $W_a$, and their relative positions for different $a\in\mathbb{D}$.

%In Section 8 we describe the C$^*$-algebra generated by $C_a$ (and $C_a^*$), and its commutant algebra $\{C_a,C_a^*\}'$. 

\section{Eigenspaces of $C_a$ and $R_a$}

Let us first  characterize the two eigenspaces $N(C_a\pm I)$ of $C_a$. Note that $\varphi_a$ has a unique fixed point $\omega_a$ inside $\mathbb{D}$, which is given by 
\begin{equation}\label{punto fijo}
\omega_a=\frac{1}{\bar{a}}\{1-\sqrt{1-|a|^2}\}=\frac{a}{|a|^2}\{1-\sqrt{1-|a|^2}\} \ \hbox{ if } a\ne 0, \ \hbox{ and } \omega_0=0.
\end{equation}
Accordingly, for $a\in\mathbb{D}$ there is a unique $\Omega_a\in\mathbb{D}$ such that the unique fixed point of $\varphi_{\Omega_a}$ is $a$, i.e., $\omega_{\Omega_a}=a$. It is given by
\begin{equation}\label{pre punto fijo}
\Omega_a=\frac{2a}{1+|a|^2}.
\end{equation}
\begin{rem}
Note the following (straightforward) formulas:
\begin{equation}\label{phis 1}
\varphi_{\omega_a}\circ \varphi_{a}=-\varphi_{\omega_a}.
\end{equation}
and 
\begin{equation}\label{phis 2}
\varphi_{a}\circ \varphi_{\omega_a}=-\varphi_{-\omega_a}.
\end{equation}

\end{rem}
\begin{teo} \label{teo 22}
For $a\in\mathbb{D}$, we have that 
$$
N(C_a-I)=\{f\in\h: C_{\omega_a}f\in\e\}=C_{\omega_a}\e
$$
and 
$$
N(C_a+I)=\{g\in\h: C_{\omega_a}g\in\o\}=C_{\omega_a}\o.
$$
\end{teo}
\begin{proof}
The formula (\ref{phis 1}) implies that 
 if $f$ is even, then
$$
f(\varphi_{\omega_a})=f(-\varphi_{\omega_a})=f(\varphi_{\omega_a}\circ\varphi_{a})=C_{a} f(\varphi_{\omega_a}),
$$
i.e., $C_{\omega_a}f\in N(C_a-I)$. The computation for odd elements is similar. Thus $C_{\omega_a}\e\subset N(C_a-I)$ and $C_{\omega_a}\o\subset N(C_a+I)$.

Conversely, if $f\in N(C_a-I)$, then $f=C_{\omega_a}h$ for a unique $h$ (equal to $C_{\omega_a}f$). If $h=h_e+h_o$, for $h_e\in\e$ and $h_o\in\o$. Then $f=C_{\omega_a}(h_e)+C_{\omega_a}(h_o)$. Since $C_{\omega_a}h_o\in N(C_a+I)$, it must be $h_o=0$, and $f\in C_{\omega _a}\e$. The assertion for $N(C_a+I)$ is proved analogously.

\end{proof}
Note that the idempotents (or non orthogonal projections)   corresponding to the direct sum decomposition $\h=N(C_a-I)\dot{+}N(C_a+I)$ are
\begin{equation}\label{idempotentes a}
P_{N(C_a-I) \parallel N(C_a+I)}=\frac12(C_a+I) \ \hbox{  and } \ P_{N(C_a+I)\parallel N(C_a-I)}=\frac12(I-C_a).
\end{equation}
The above result implies the following:
\begin{coro}\label{idempotentes a bis}
$$
\frac12(C_a+I)=C_{\omega_a}P_\e C_{\omega_a} \ \hbox{ and } \ \  \frac12(I-C_a)=C_{\omega_a}P_\o C_{\omega_a},
$$
i.e., $C_a=C_{\omega_a}C_0C_{\omega_a}$.
\end{coro}
\begin{proof}
$C_{\omega_a}P_\e C_{\omega_a}$ is an idempotent operator whose range is
$$
R(C_{\omega_a}P_\e C_{\omega_a})=C_{\omega_a}R(P_\e)=C_{\omega_a}\e=N(C_a-I),
$$
and whose nullspace is 
$$
N(C_{\omega_a}P_\e C_{\omega_a})=C_{\omega_a}N(P_\e)=C_{\omega_a}\o=N(C_a+I).
$$
The other verification is similar.
$\frac12(I+C_a)=C_{\omega_a}P_\e C_{\omega_a}$ means that 
$$
C_a=2C_{\omega_a}P_\e C_{\omega_a}-I=C_{\omega_a}(2P_\e-I)C_{\omega_a}=C_{\omega_a}C_0 C_{\omega_a}.
$$ 
\end{proof}

Let us relate the eigenspaces of $C_a$ with those of $R_a$.
First note the following:
\begin{lem}\label{cuentita}
$$
C_a |1-\bar{a}z|^{1/2}=\frac{(1-|a|^2)^{1/2}}{|1-\bar{a}z|^{1/2}}
$$
and
$$
C_a \frac{1}{|1-|a|^2|^{1/2}}=\frac{|1-\bar{a}z|^{1/2}}{(1-|a|^2)^{1/2}}.
$$
\end{lem}
\begin{proof}
Direct computation.
\end{proof}
\begin{prop}\label{autoespacios Ra}

\noindent

\begin{enumerate}
\item
$f\in N(R_a-I)$ if and only if $|1-\bar{a}z|^{1/2} f\in N(C_a-I)$.
\item
$g\in N(R_a+I)$ if and only if $|1-\bar{a}z|^{1/2} g\in N(C_a+I)$.
\end{enumerate}
\end{prop}
\begin{proof}
We only prove assertion 1., assertion 2. is similar. Let $f\in\h$ such that $R_af=f$. i.e.,
$$
(1-|a|^2)^{1/2}\frac{1}{|1-\bar{a}z|} C_af=f.
$$
As in the above Lemma, $C_a|1-\bar{a}z|^{1/2}=\frac{(1-|a|^2)^{1/2}}{|1-\bar{a}z|^{1/2}}$. Then
$$
f=(1-|a|^2)^{1/2}\frac{1}{|1-\bar{a}z|} C_af=\frac{1}{|1-\bar{a}z|^{1/2}} C_a( |1-\bar{a}z|^{1/2}) C_af.
$$
Note the elementary multiplicative property of  $C_a$: if $g,h\in\h$ such that $gh\in\h$, then $C_a(gh)=C_ag C_ah$. Then we have
$$
f=\frac{1}{|1-\bar{a}z|^{1/2}} C_a( |1-\bar{a}z|^{1/2}f),
$$
i.e., $C_a(|1-\bar{a}z|^{1/2}f)=|1-\bar{a}z|^{1/2}f$.

The other inclusion: suppose that $C_af=f$. Then (using Lemma \ref{cuentita}, and the multiplicative property of $C_a$)
$$
R_a(\frac{1}{|1-\bar{a}z|^{1/2}}f)=\frac{(1-|a|^2)^{1/2}}{|1-\bar{a}z|} C_a(\frac{1}{|1-\bar{a}z|^{1/2}} f)=\frac{(1-|a|^2)^{1/2}}{|1-\bar{a}z|}C_a(\frac{1}{|1-\bar{a}z|^{1/2}})C_a f
$$
$$
=\frac{1}{|1-\bar{a}z|^{1/2}} f,
$$
which completes the proof.
\end{proof}

We may thus summarize these facts:
\begin{coro}

\noindent

\begin{enumerate}
\item
$N(R_a-I)=\{ f\in\h: C_{\omega_a}(|1-\bar{a}z|^{1/2}f) \in \e\}=M_{\frac{1}{|1-\bar{a}z|^{1/2}}} C_{\omega_a}\e$.
\item
$N(R_a+I)=\{ g\in\h: C_{\omega_a}(|1-\bar{a}z|^{1/2}g) \in \o\}=M_{\frac{1}{|1-\bar{a}z|^{1/2}}} C_{\omega_a}\o$.

\end{enumerate}
\end{coro}

\begin{rem}\label{que cuerno}
Note that
\begin{equation}\label{corno}
C_{\omega_a}(|1-\bar{a}z|)=(1-|a|^2)^{1/2}\frac{|1+\bar{\omega}_az|}{|1-\bar{\omega}_az|}.
\end{equation}
Thus $C_{\omega_a}(|1-\bar{a}z|^{1/2})=(1-|a|^2)^{1/4}\frac{|1+\bar{\omega}_az|^{1/2}}{|1-\bar{\omega}_az|^{1/2}}$, and
\begin{equation}\label{25 mejorado}
N(R_a-I)=\{f\in\h: \frac{|1+\bar{\omega}_az|^{1/2}}{|1-\bar{\omega}_az|^{1/2}}f(\varphi_{\omega_a})\in\e\}.
\end{equation}
Indeed,
$$
C_{\omega_a}(|1-\bar{a}z|)=|1-\bar{a}\frac{\omega_a-z}{1-\bar{\omega}_az}|=\frac{|1-\bar{a}\omega_a+z(\bar{a}-\bar{\omega}_a)|}{|1-\bar{\omega}_az|}.
$$
Note that if $r=(1-|a|^2)^{1/2}$ then 
$$
\bar{a}\omega_a=1-r \ \hbox{ and } \ \bar{a}-\bar{\omega}_a=\frac{1}{a}(|a|^2-(1-r))=\frac{1}{a}(r-r^2).
$$
Then, in the above computation we have
$$
\frac{|r+\frac{1}{a}(r-r^2)z)|}{|1-\bar{\omega}_az|}=r \frac{|1+\frac{1}{a}(1-r)z)|}{|1-\bar{\omega}_az|}=r\frac{|1+\bar{\omega}_az|}{|1-\bar{\omega}_az|}.
$$
Similarly $C_{\omega_a}(\frac{1}{|1-\bar{a}z|})=\frac{1}{(1-|a|^2)^{1/2}}\frac{|1-\bar{\omega}_az|}{|1+\bar{\omega}_az|}$,
and
\begin{equation}\label{25 mejorado bis}
N(C_{\omega_a}+I)=\{g\in\h: \frac{|1-\bar{\omega}_az|^{1/2}}{|1+\bar{\omega}_az|^{1/2}}g(\varphi_{\omega_a})\in\o\}.
\end{equation}
\end{rem}

\begin{rem}\label{otras simetrias}
We shall consider other  symmetries which interact with $C_a$:

\begin{enumerate}
\item
Let $V$ be the symmetry given by $Vf(z)=f(\bar{z})$. The eigenspaces of $V$ are
$$
N(V-I)=\{f=\sum_{m\in\mathbb{Z}}a_m z^m: a_m=a_{-m}\}
$$
and
$$
N(V+I)=\{f=\sum_{m\in\mathbb{Z}}a_m z^m: a_m=-a_{-m}\}.
$$
This symmetry has the following commutation relation with $C_a$:
\begin{equation}\label{Ca y V}
VC_a=C_{\bar{a}}V,
\end{equation}
and in particular, if $a\in(-1,1)$, $V$ and $C_a$ commute.
Indeed,
$$
VC_af(z)=Vf(\frac{a-z}{1-\bar{a}z})=f(\frac{a-\bar{z}}{1-\bar{a}\bar{z}})=f(\frac{a-\frac{1}{z}}{1-\bar{a}\frac{1}{z}})=f(\frac{1-\bar{a}z}{a-z})=f\left(\overline{{\varphi}_{\bar{a}}(z)}\right)
$$
$$
=C_{\bar{a}}Vf(z).
$$
If one identifies $L^2(\mathbb{T})$ with $L^2(-\pi,\pi)$, by means of $z=e^{it}$, then $N(V-I)$ consists of elements $f$ such that $f(t)=f(-t) \ pp$, i.e., are even in the parameter $t$. Similarly $N(V+I)$ consists of elements which are odd in the parameter $t$.
\item
Recall the symmetry $W_a$ from (\ref{Wa}), defined in $\h$ as 
$$
W_a=M_{\frac{\sqrt{1-|a|^2}}{1-\bar{a}z}}C_a, \ \hbox{i.e., } \ W_af=\frac{\sqrt{1-|a|^2}}{1-\bar{a}z} f(\varphi_a).
$$
Note that $W_a$ leaves $\h^+$ invariant, in fact $W_a\h^+=\h^+$. The restriction ${\bf W}_a=W_a\Big|_{\h^+}$ was considered before, see for instance \cite{cowenmccluer} (Exercise 2.1.9), where it was shown that it acts isometrically  in $H^p(\mathbb{D})$ for $1<p<\infty$. It is easy to see that $W_a$ is a symmetry in $\h$:
$$
W_a^2f=(1-|a|^2)^{1/2} W_a(\frac{1}{1-\bar{a}z} f(\varphi_a))=\frac{1-|a|^2}{1-\bar{a}z} \frac{1}{1-\bar{a}\frac{a-z}{1-\bar{a}z}} f=f
$$
and
$$
W_a^*f=C_a^*M_{\frac{\sqrt{1-|a|^2}}{1-\bar{a}z}}^*f=M_{\frac{1-|a|^2}{|1-\bar{a}z|^2}} C_a M_{\frac{1-|a|^2}{1-a\bar{z}}}f=\frac{(1-|a|^2)^{3/2}}{|1-\bar{a}z|^2} \frac{1}{1-a\frac{\bar{a}-\bar{z}}{1-a\bar{z}}} f(\varphi_a)
$$
$$
=\frac{(1-|a|^2)^{1/2}}{1-\bar{a}z} f(\varphi_a)=W_af.
$$

Note that  $W_a=M_{\frac{\sqrt{1-|a|^2}}{1-\bar{a}z}}C_a$ and $R_a=M_{\frac{\sqrt{1-|a|^2}}{|1-\bar{a}z|}}C_a=C_a M_{\frac{1-\bar{a}z}{\sqrt{1-|a|^2}}}$ satisfy
$$
W_a R_a=M_{\frac{|1-\bar{a}z|}{1-\bar{a}z}} \ \hbox{ and } \ R_aW_a=M_{\frac{1-\bar{a}z}{|1-\bar{a}z|}}.
$$
\end{enumerate}

\end{rem}

For $a\in\mathbb{D}$, denote by $(1-\bar{a}z)^{1/2}$ the (non continuous) bounded Borel function 
$$
(1-\bar{a}z)^{1/2}:=\exp \left(\frac12 \log(1-\bar{a}z)\right), \ \hbox{ for } z\in\mathbb{T}
$$
where $\log$ is the usual determination of the lograrithm: $\log(z)=\ln |z|+i \arg(z)$, with $\arg(z)\in[-\pi,\pi)$. Similarly as with the eigenspaces of $R_a$, we have:
\begin{prop}
Let $a\in\mathbb{D}$. Then
$$
N(W_a-I)=\{f\in \h: \frac{f}{(1-\bar{a}z)^{1/2}}\in N(C_a-I)\},
$$
and 
$$
N(W_a+I)=\{g\in \h: \frac{g}{(1-\bar{a}z)^{1/2}}\in N(C_a+I)\},
$$
\end{prop}
\begin{proof}
The proof is similar as in Proposition \ref{autoespacios Ra}. We sketch the argument.
If $C_af=f$, then
$$
W_a \frac{f}{(1-\bar{a}z)^{1/2}}=\frac{(1-|a|^2)^{1/2}}{1-\bar{a}z} f(\varphi_a) \left( \frac{1}{1-\bar{a}\frac{a-z}{1-\bar{a}z}}\right)^{1/2}=\frac{(1-|a|^2)^{1/2}}{1-\bar{a}z}  \left(\frac{1-\bar{a}z}{1-|a|^2}\right)^{1/2}f=f.
$$
For the reverse inclusion, note that 
$C_a(1-\bar{a}z)^{1/2}=\frac{(1-|a|^2)^{1/2}}{(1-\bar{a}z)^{1/2}}$. Therefore if $W_af=f$, then 
$$
f=\frac{(1-|a|^2)^{1/2}}{1-\bar{a}z} C_af=\frac{1}{(1-\bar{a}z)^{1/2}}\frac{(1-|a|^2)^{1/2}}{(1-\bar{a}z)^{1/2}}C_af=\frac{1}{(1-\bar{a}z)^{1/2}}C_a(1-\bar{a}z)^{1/2} C_af
$$
$$
=\frac{1}{(1-\bar{a}z)^{1/2}}C_a\left((1-\bar{a}z)^{1/2}f\right),
$$
i.e.,
$$
(1-\bar{a}z)^{1/2}f=C_a\left((1-\bar{a}z)^{1/2}f\right).
$$
\end{proof}
\section{The maps $a\mapsto C_a$ and $a\mapsto R_a$}
Note that $C_a$ restricts to the Hardy space $\h^+=H^2(\mathbb{D})$. Let us denote by $$
{\bf C}_a:=C_a\Big|_{\h^+}\in\b(\h^+).
$$
In \cite{berkson}, E. Berkson proved that in the Hardy space $\h^+$, if $a\ne b$, one has 
$$
\|{\bf C}_a-{\bf C}_b\|\ge\frac{1}{\sqrt2}.
$$
Therefore,  $\|C_a-C_b\|\ge\frac{1}{\sqrt2}$. However, note that 
$$
\||C_a|-|C_b|\|=\|(1-|a|^2)^{1/2} M_{\frac{1}{|1-\bar{a}z|}}-(1-|b|^2)^{1/2} M_{\frac{1}{|1-\bar{b}z|}}\|=\|M_{\frac{(1-|a|^2)^{1/2} }{|1-\bar{a}z|}-\frac{(1-|b|^2)^{1/2} }{|1-\bar{b}z|}}\|
$$
$$
=\|\frac{(1-|a|^2)^{1/2} }{|1-\bar{a}z|}-\frac{(1-|b|^2)^{1/2} }{|1-\bar{b}z|}\|_\infty=\sup_{|z|=1}\Big|\frac{(1-|a|^2)^{1/2} }{|1-\bar{a}z|}-\frac{(1-|b|^2)^{1/2} }{|1-\bar{b}z|}\Big|:=\gamma_{a,b}.
$$
It is clear that $\gamma_{a,b}\to 0$ if $b\to a$, and therefore the map
$$
\mathbb{D}\ni a\mapsto |C_a|\in\b(\h)
$$
is continuous.

Note the following elementary estimation
$$
\frac{1}{\sqrt2}\le \|C_a-C_b\|=\|R_a M_{\frac{(1-|a|^2)^{1/2}}{|1-\bar{a}z|}}-R_b M_{\frac{(1-|b|^2)^{1/2}}{|1-\bar{b}z|}}\| 
$$
$$
\le\| R_a M_{\frac{(1-|a|^2)^{1/2}}{|1-\bar{a}z|}}-R_a M_{\frac{(1-|b|^2)^{1/2}}{|1-\bar{b}z|}}\|+\|R_a M_{\frac{(1-|b|^2)^{1/2}}{|1-\bar{b}z|}}-R_b M_{\frac{(1-|b|^2)^{1/2}}{|1-\bar{b}z|}}\|
$$
$$
\le \|R_a-R_b\| \|\frac{(1-|a|^2)^{1/2}}{|1-\bar{a}z|}\|_\infty+\|\frac{(1-|a|^2)^{1/2}}{|1-\bar{a}z|}-\frac{(1-|b|^2)^{1/2}}{|1-\bar{b}z|}\|_\infty=\|R_a-R_b\|\left(\frac{1+|a|}{1-|a|}\right)^{1/2}+\gamma_{a,b}.
$$
By symmetry,
$$
\frac{1}{\sqrt2}\le \|R_a-R_b\| \min\Big\{\frac{1+|a|}{1-|a|},\frac{1+|b|}{1-|b|}\Big\}^{1/2} +\gamma_{a,b}.
$$
In particular, since $\gamma_{a,b}\to 0$ if $b\to a$, the map 
$$
\mathbb{D}\ni a\mapsto R_a\in \b(\h)
$$ 
is not continuous.

\section{Relative position of the eigenspaces}

Denote by $P^+$ the orthogonal projection onto $\h^+\subset\h$. We already noted that $C_a\h^+\subset\h^+$ (in fact, $C_a\h^+=\h^+$). Note that $\h^+$ is not invariant for $C_a^*$, but for a one dimensional space. Denote by $\h^-:=(\h^+)^\perp$.

If $g\in\h^-$, then $\bar{g}\in\h^+$. Denote by ${\bf \bar{g}}$ the extension of $\bar{g}$ to an analytic function in $\mathbb{D}$.

\begin{lem}\label{lema 51}
$$
C_a\h^-=\langle 1\rangle \oplus \h^-.
$$
More precisely, if $g\in\h^-$, then 
$$
C_ag=\overline{{\bf \bar{g}}(a)}+h, \ \hbox{ for } \ h\in\h^-.
$$
\end{lem} 
\begin{proof}
Let $g\in\h^-$ and $n\ge 1$. Then
$$
\langle C_a g,z^n \rangle=\langle g,C_a^*z^n\rangle=(1-|a|^2)\langle g, \frac{1}{|1-\bar{a}z|^2}\left(\frac{a-z}{1-\bar{a}z}\right)^n\rangle.
$$
Note that if $|z|=1$, then $\frac{1}{|1-\bar{a}z|^2}=\frac{z}{(1-\bar{a}z)(z-a)}$. Then
$$
\langle C_a g,z^n \rangle=-(1-|a|^2)\langle g, \frac{z(a-z)^{n-1}}{(1-\bar{a}z)^{n+1}}\rangle=0.
$$
Moreover,
$$
\langle C_a g,1 \rangle=(1-|a|^2)\langle g, \frac{z}{(1-\bar{a}z)(z-a)}\rangle =(1-|a|^2)\langle \frac{\bar{z}}{\bar{z}-\bar{a}}, \frac{\bar{g}}{1-\bar{a}z}\rangle.
$$
Note that (for $|z|=1$),  $\frac{\bar{z}}{\bar{z}-\bar{a}}=\frac{\frac{1}{z}}{\frac{1}{z}-\bar{a}}=\frac{1}{1-\bar{a}z}=k_a(z)$.
Then, since $\bar{g}\in\h^+$,  
$$
\langle C_a g,1 \rangle=(1-|a|^2)\langle k_a, \frac{\bar{g}}{1-\bar{a}z}\rangle=(1-|a|^2)\overline{\frac{{\bf \bar{g}}(a)}{1-|a|^2}}=\overline{{\bf \bar{g}}(a)},
$$ 
which completes the proof. 
\end{proof}

As a consequence, we have the following:
\begin{prop}\label{invarianza + o -}

\noindent

\begin{enumerate}
\item
Let $f\in L^2(\mathbb{T})$,  $f\in N(C_a-I)$, and $f=f_++f_-$, with $f_+\in\h^+$ and $f_-\in \h^-$. Then
$$
f_+, f_-\in N(C_a-I).
$$
\item
Let $f\in N(C_a+I)$, and $f=f_++f_-$ s above. Then
$$
C_a f_+=-f_+-\overline{{\bf \bar{f}_-}(a)} \ \ \hbox{ and } \ \ C_a f_-=-f_-+\overline{{\bf \bar{f}_-}(a)}.
$$
\end{enumerate}
\end{prop}
\begin{proof}
 Suppose that $C_af=f$. Then $f_++f_-=C_af_++C_af_-$. 
Note that $C_af_+\in \h^+$. By  Lemma \ref{lema 51}, $C_af_-=g+k 1$, for $g\in \h^-$ and $k\in\mathbb{C}$. Thus
$$
C_af_++g=f_+-k 1+f_- \ \implies C_af_+=f_+-k 1.
$$
Applying $C_a$ to this equality, using that constant functions belong to $N(C_a-I)$,  we get $f_+=C_af_+-k 1$, and thus $k=0$. Therefore $C_af_+=f_+$ and $C_af_-=f_-$. 

 For the case of  $N(C_a+I)$, we can proceed as above, using Lemma \ref{lema 51}: $f\in N(C_a+I)$, $f=f_++f_-$, and thus $C_af_++C_af_-=-f_+-f_-$. Then $C_af_-=g +\overline{ {\bf \bar{f}_-}(a)}$, and therefore $C_af_+=-f_+-\overline{ {\bf \bar{f}_-}(a)}$ and $C_af_-=g=-f_-+\overline{ {\bf \bar{f}_-}(a)}$. Only this time we cannot deduce that  $\overline{ {\bf \bar{f}_-}(a)}$ is zero.
\end{proof}
\begin{rem}
There is an easy example of $f\in N(C_a+I)$ that shows that $C_af_-$ can have a constant coefficient (and  therefore does not belong to $\h^-$). Pick $f=\frac{1}{\varphi_{\omega_a}}=\frac{1-\bar{\omega}_az}{\omega_a-z}$. Since $f=C_{\omega_a}(\frac{1}{z})$, and $\frac{1}{z}$ is an odd element of $\h$, it follows that $f\in N(C_a+I)$. Also note that
$$
f=\bar{\omega}_a+\frac{\bar{\omega}_a^2-1}{z-\omega_a},
$$
with $f_-=\frac{\bar{\omega}_a^2-1}{z-\omega_a}$: 
$$
\bar{{\bf f}}_-=\frac{\omega_a^2-1}{\frac{1}{z}-\bar{\omega}_a}=\frac{\omega_a^2-1}{1-\bar{\omega}_az}z\in\h^+, \ \hbox{ and vanishes at } z=0.
$$
Clearly $\overline{\bar{\bf f}_-(a)}=\frac{\omega_a^2-1}{a-\bar{\omega}_a}a\ne 0$. 
\end{rem}

Note that $V\h^-\subset\h^+\ominus\langle 1\rangle$ and $V\h_+=\langle 1 \rangle \oplus \h^-$.
\begin{rem}\label{hombre al agua}
In \cite{disco y composicion}  the eigenspaces of the restriction ${\bf C}_a:=C_a\big|_{\h^+}$ of $C_a$ to $\h^+$ were considered. It was shown  that if $a\ne b$ in $\mathbb{D}$, then
$$
N({\bf C}_a-I)\cap N({\bf C}_b-I)=\mathbb{C}1 \ \hbox{ and } \ N({\bf C}_a+I)\cap N({\bf C}_b+I)=\{0\}.
$$
(\cite{disco y composicion} Theorem 5.6).
\end{rem}

\begin{teo}\label{44}
If $a\ne b$ in $\mathbb{C}$,
\begin{enumerate}
\item
$N(C_a-I)\cap N(C_b-I)=\mathbb{C}1$.
\item
$N(C_a+I)\cap N(C_b+I)=\{0\}$.
\end{enumerate}
\end{teo}
\begin{proof} The first assertion: clearly $\mathbb{C}1\subset N(C_a-I)\cap N(C_b-I)$. 
Let $f\in N(C_a-I)\cap N(C_b-I)$, and $f=f_++f_-$ with $f_+\in\h^+$ and $f_-\in \h^-$. Then by Proposition \ref{invarianza + o -}.1, $f_+, f_-\in N(C_a-I)\cap N(C_b-I)$. Then (restricting to $\h^+$) $f_+\in N({\bf C}_a-I)\cap N({\bf C}_b-I)=\mathbb{C}1$. Also note that $Vf_-\in\h^+$. It is clear that
$$
VN(C_a-I)=N(VC_aV-I)=N(C_{\bar{a}}-I),
$$
and similarly for $b$. Then $Vf_-\in N(C_{\bar{a}}-I)\cap N(C_{\bar{b}}-I)=N({\bf C}_{\bar{a}}-I)\cap N({\bf C}_{\bar{b}}-I)=\mathbb{C}1$. However $Vf_-\perp 1$, and thus $f_-=0$.

Let now $f\in N(C_a+I)\cap N(C_b+I)$. To treat this case, we shall make two reductions. First, that we can consider the case $b=0$. To this effect, note the following elementary fact, which follows from direct computations. If $d,b\in\mathbb{D}$, then
\begin{equation}\label{cacbca}
C_dC_bC_d=C_{d\bullet b} \ , \ \hbox{ where } \ d\bullet b:=\frac{2d-b-d^2\bar{b}}{|d|^2-d\bar{b}-\bar{d}b+1}.
\end{equation}
 Therefore, if $b\ne 0$, and we  choose $d=1-\sqrt{1-|b|^2}$, which is the unique element in $\mathbb{D}$ such that $d\bullet b=0$, we have that the statement
$$
N(C_a+I)\cap N(C_b+I)=\{0\}
$$
is equivalent to
$$
\{0\}=C_d\left(N(C_a+I)\cap N(C_b+I)\right)=N(C_dC_aC_d+I)\cap N(C_dC_bC_d+I)
$$
$$
=N(C_{d\bullet a}+I)\cap N(C_0+I).
$$
Next, we can further reduce to the case $a=r\in(0,1)$. Indeed, if $a=re^{i\theta}$, consider the unitary operator $U_\theta$ in $\h$ given by
$U_\theta f(z)=f(e^{-i\theta}z)$. Then it is elementary to verify that
\begin{equation}\label{Utheta}
U_\theta C_r U_\theta^*=C_{e^{i\theta}r}=C_a.
\end{equation}
Thus
$$
N(C_a+I)\cap N(C_0+I)=\{0\}
$$
is equivalent to 
$$
\{0\}=U_{-\theta}(N(C_a+I)\cap N(C_0+I))=N(U_{-\theta} C_a U_\theta+I)\cap N(U_{-\theta} C_0 U_\theta+I)
$$
$$
=N(C_r+I)\cap N(C_0+I).
$$

Then we may assume that $C_rf=f$ and $C_0f=f$ (i.e., $f$ is odd). Decompose $f=f_++f_-$, with $f_+\in\h^+$ and $f_-\in\h^-$. Clearly $f_+, f_-$ are also odd, and then $C_0f_+=-f_+$ and $C_0f_-=-f_-$ . By Proposition \ref{invarianza + o -} we know that  $C_rf_+=-f_+-k$ and $C_rf_-=-f_-+k$, for some constant $k$.
Since $f_+\in\h^+$ we may regard it as a function in $\mathbb{D}$, and evaluate $C_rf_+=-f_+-k$   at $z=0$ (since $f_+$ is odd, $f_+(0)=0$):
\begin{equation}\label{bola1}
f_+(r)=f_+(\varphi_r(0))=C_rf_+(0)=-f_+(0)-k=-k.
\end{equation}
Note also that the identity $C_rf_+=-f_+-k$ implies that $f_+-\frac{k}{2}\in N(C_r-I)$.
Therefore, $h=C_{\omega_r}(f_+-\frac{k}{2})$ is an odd function in $\mathbb{D}$. 
In particular, $h(\omega_r)=-h(-\omega_r)$.
On one hand 
$$
h(\omega_r)=f_+(\varphi_{\omega_r}(\omega_r))-\frac{k}{2}=f_+(0)-\frac{k}{2}=-\frac{k}{2}.
$$
On the other hand
$$
-h(-\omega_r)=-f_+(\varphi_{\omega_r}(-\omega_r))+\frac{k}{2}.
$$
Note that $\varphi_{\omega_r}(-\omega_r)=\frac{2\omega_r}{1+\omega_r^2}=\Omega_{\omega_r}=r$ (recall (\ref{pre punto fijo})). Then the above equals
$$
-h(-\omega_r)=-f_+(r)+\frac{k}{2}.
$$
Then $h(\omega_r)=-h(-\omega_r)$ means that $-\frac{k}{2}=-f_+(r)+\frac{k}{2}$. That is $f_+(r)=k$, which together with  (\ref{bola1}) imply that $k=0$. Then $C_rf_+=-f_+$ and $C_rf_-=-f_-$. Since $f_+\in\h^+$, by  Theorem 5.6 in \cite{disco y composicion} (Remark \ref{hombre al agua} above), we have that $f_+=0$. 

We may use the isometry $V$ ($Vf(z)=f(\bar{z})$), which commutes with $C_r$ (see (\ref{Ca y V})), because $r$ is real. Clearly $Vf_-\in\h^-$. Then 
$$
C_rVf_-=VC_rf_-=-Vf_-,
$$ 
and thus $f_-=0$.
\end{proof}

Let us now consider intersections involving  the orthogonal complements of the eigenspaces of $C_a$ and $C_b$. 
We would like  to consider arbitrary $a\ne b$. We were not able to achieve this, save for some special cases. We examine first the case $b=0$:
\begin{teo}\label{posicion 1}
Let $a\ne 0$. Then 
\begin{enumerate}
\item
$N(C_a-I)^\perp\cap \e=\{0\}=N(C_a-I)\cap \o$.
\item
$N(C_a+I)^\perp\cap \o=\{0\}=N(C_a+I)\cap\e$.
\end{enumerate}
\end{teo}
\begin{proof}
We prove first the left hand intersection in the first assertion.
Let $f$ even such that $f\perp N(C_a-I)=C_{\omega_a}(\e)$. Put $f=f_++f_-$ with $f_+\in\h^+$ and $f_-\in\h^-$. Clearly both $f_+$ and $f_-$ are even. Then
$$
0=\langle f_++f_-,C_{\omega_a}z^{2m}\rangle.
$$
For $m\ge 0$, since $C_{\omega_a}(z^{2m})={\bf C}_{\omega_a}z^{2m}\in\h^+$, we have
$$
0=\langle f_+,{\bf C}_{\omega_a}z^{2m}\rangle,
$$
and thus $f_+\in N({\bf C}_0-I)\cap N({\bf C}_{\omega_a}-I)^\perp$. In \cite{disco y composicion} (Theorem 5.7) it was proven that this intersection is trivial. Then $f_+=0$.
Thus, for $m<0$, since $\varphi_{\omega_a}(z)$ has modulus $1$ for $z\in\mathbb{T}$,
$$
0=\langle f_-, C_{\omega_a}z^{2m}\rangle \implies 0=\langle \overline{f}_-,\overline{ C_{\omega_a}z^{2m}}\rangle=\langle\overline{f}_-, \left(\left(\varphi_{\omega_a}(z)\right)^{2m}\right)^{-}\rangle=\langle \overline{f}_-, \varphi_{\omega_a}^{-2m}\rangle,
$$
i.e., the even analytic element $\overline{f}_-$ is orthonogonal to $C_{\omega_a}z^{k}$ for $k\ne 0$. Clearly also $\overline{f}_-\perp 1$, so that by the same result cited above. $f_-=0$.

The right hand intersection in the first assertion is  simpler. Suppose  $f=f_++f_-$  is odd and satisifies $C_af=f$. Then by Proposition \ref{invarianza + o -}, we have that $f_+$ and $f_-$ are odd and belong to $N(C_a-I)$. Then $f_+$ and $\overline{f}_-$ are odd (analytic) and belong to $N({\bf C}_a-I)$. Therefore, again using Theorem 5.7 in \cite{disco y composicion}. we have that $f_+, f_-=0$.

The proof of the second assertion is similar.
\end{proof}

There is an analogous result, which can be proved similarly:
\begin{teo}\label{posicion 2}
Let $a\ne 0$. Then 
\begin{enumerate}
\item
$N(C_a+I)^\perp\cap \e=\{0\}$.
\item
$N(C_a-I)^\perp\cap \o=\{0\}$.
\end{enumerate}
\end{teo}
We could not solve the case when $b\ne 0$. Let us state the following partial observations:
\begin{rem}
Let us briefly elaborate on  the intersection $\s:=N(C_a-I)\cap N(C_b-I)^\perp$, for arbitrary $a\ne b$ in $\mathbb{D}$. We compute $R_b\s$. First since $R_b$ is a symmetry
$$
R_b\left(N(C_b-I)^\perp\right)=\left(R_b N(C_b-I)\right)^\perp=\left(N(R_bC_bR_b-I)\right)^\perp.
$$
Using (\ref{boludez}) this equals
$$
N(C_b ^*-I)^\perp=R(C_b-I)=N(C_b+I).
$$
On the other hand, $R_b(N(C_a-I))=N(R_bC_aR_b-I)$, and using both expressions of $R_b$ in (\ref{Ra}), and (\ref{cacbca}), we get
$$
R_bC_aR_b=M_{\frac{\sqrt{1-|b|^2}}{|1-\bar{b}z|}}C_bC_aC_bM_{\frac{|1-\bar{b}z|}{\sqrt{1-|b|^2}}}=M_{\frac{1}{|1-\bar{b}z|}}C_{b\bullet a}M_{|1-\bar{b}z|}.
$$
Then 
$$
R_b\s=N(C_b+I)\cap M_{\frac{1}{|1-\bar{b}z|}} N(C_{b\bullet a}-I).
$$
Or equivalently, applying $M_{\frac{|1-\bar{b}z|}{(1-|b|^2)^{1/2}}}$, 
$$
C_b\s=M_{|1-\bar{b}z|}R_b\s=N(C_{b\bullet a}-I)\cap M_{|1-\bar{b}z|} N(C_b+I).
$$

\end{rem}
\section{Position of $N(C_a-I)$ and $N(C_a+I)$.}\label{subseccion 41}
In this subsection we treat the special case of (the position of)  $N(C_a-I)$ and $N(C_a+I)$. 
First note that $N(C_a-I)\cap N(C_a+I)=\{0\}$ and 
$$
N(C_a-I)^\perp\cap N(C_a+I)^\perp=\langle N(C_a-I)\vee N(C_a+I)\rangle^\perp=\h^\perp=\{0\}.
$$
Suppose $a\ne 0$. As in the proof of Theorem \ref{44}, using the unitary $U_\theta$ (for $a=re^{i\theta}$, see (\ref{Utheta}))  we can reduce to the case $a=r\in(0,1)$. 

In order to further consider these subspaces we shall use their associated orthogonal projections $P_{N(C_a-I)}$ and $P_{N(C_a+I)}$. For an arbitrary idempotent $Q$, recall the formulas (see for instance T. Ando \cite{ando})
\begin{equation}\label{ando 1}
P_{R(Q)}=Q(Q+Q^*-I)^{-1} \ \hbox{ and } \ P_{N(Q)}=(I-Q)(I-Q-Q^*)^{-1}
\end{equation}
In our case $Q=\frac12(C_a+I)$ and thus
$$
P_{N(C_a-I)}=(C_a+I)\left(C_a+C_a^*\right)^{-1}=(C_a+I)\left(C_a+M_{\frac{1-|a|^2}{|1-\bar{a}z|^2}}C_a\right)^{-1}
$$
$$
=(C_a+I)\left( M_{1+\frac{1-|a|^2}{|1-\bar{a}z|^2}} C_a\right)^{-1}=(C_a+I)C_aM_{\psi_a}=(I+C_a)M_{\psi_a}
$$
where $\psi_a(z):=\left(1+\frac{1-|a|^2}{|1-\bar{a}z|^2}\right)^{-1}$, 
and similarly
$$
P_{N(C_a+I)}=(I-C_a)(C_a+C_a^*)^{-1}=(I-C_a)M_{\psi_a}.
$$
\begin{prop}\label{norma 1}
If $a\in\mathbb{D}$, $a\ne 0$, then
$$
\|P_{N(C_a-I)}-P_{N(C_a+I)}\|=1
$$
\end{prop}
\begin{proof}
If $a=re^{i\theta}$, Then $C_a=U_{\theta}C_rU_{-\theta}$, $C_a^*=U_{\theta}C_r^*U_{-\theta}$, and thus 
$$
\|P_{N(C_a-I)}-P_{N(C_a+I)}\|=\|U_{-\theta}P_{N(C_r-I)}U_{\theta}- U_{-\theta}P_{N(C_r+I)}U_{\theta}\|
$$
$$
=\|U_{-\theta}(P_{N(C_r-I)}-P_{N(C_r+I)})U_{\theta}\|=\|P_{N(C_r-I)}-P_{N(C_r+I)}\|,
$$
i.e., we reduce to the case $a=r\in(0,1)$.  By the above formulas, 
$$
P_{N(C_r-I)}-P_{N(C_r+I)}=2 C_r M_{\psi_r}.
$$
Then
$$
(P_{N(C_r-I)}-P_{N(C_r+I)})^2=(P_{N(C_r-I)}-P_{N(C_r+I)})^*(P_{N(C_r-I)}-P_{N(C_r+I)})=4 M_{\psi_r}C_r^*C_rM_{\psi_r}
$$
$$
=4 M_{\psi_r}M_{\frac{1-r^2}{|1-rz|^2}} C_rC_r M_{\psi_r}=M_{4\psi_r^2 \frac{1-r^2}{|1-rz|^2}}.
$$
If we identify $L^2(\mathbb{T})\sim L^2([-\pi,\pi])$, $f(e^{it})\sim f(t)$, then
$|1-rz|^2\sim 1+r^2-2r\cos(t)$ and then
$$
4\psi_r^2 \frac{1-r^2}{|1-rz|^2}\sim \frac{(1+r^2-2r\cos(t))}{(1-r\cos(t))^2} (1-r^2).
$$
Thus, we have to compute 
$$
\sup_{t\in[-\pi,\pi]} \frac{(1+r^2-2r\cos(t))}{(1-r\cos(t))^2}=\sup_{s\in[-r,r]} \frac{1+r^2-s}{(1-s)^2}=\frac{1}{1-r^2}.
$$
Then 
$$
\|P_{N(C_r-I)}-P_{N(C_r+I)}\|^2=1.
$$
\end{proof}

\begin{rem}
In \cite{davis} Chandler Davis studied operators which are differences of projections (see also \cite{p-q}). An operator $A=P-Q$, for $P,Q$ orthogonal projections, is a selfadjoint contraction, with the additional  following spectral property: 
\begin{enumerate}
\item
For $|\lambda|<1$, one has that $\lambda\in\sigma(A)$ if and only if $-\lambda\in\sigma(A)$.  Moreover, the spectral multiplicity function of $A$ is symmetric with respect to the origin. In particular, $\lambda$ (with $|\lambda|<1$) is an eigenvalue of $A$ if and only if $-\lambda$ is also an eigenvalue of $A$, and in that case they have the same multiplicity.
\item
$N(A-I)=R(P)\cap N(Q)$ and $N(A+I)=N(P)\cap R(Q)$ may have different dimensions.
\end{enumerate}
\end{rem}
\begin{teo}\label{coro davis}
Let $a\in\mathbb{D}$, $a\ne 0$.
$$
\sigma(P_{N(C_a-I)}-P_{N(C_a+I)})=[-1,-(1-|a|^2)^{1/2}]\cup[(1-|a|^2)^{1/2},1].
$$
There are no eigenvalues. In particular
$$
N(C_a-I)\cap N(C_a+I)^\perp=\{0\}=N(C_a-I)^\perp\cap N(C_a+I).
$$
\end{teo}
\begin{proof}
Again, if $a=re^{i\theta}$, it suffices to consider the case $a=r\in(0,1)$. 
Recall from the proof of  Proposition \ref{norma 1}, that $(P_{N(C_r-I)}-P_{N(C_r+I)})^2=M_{\psi_r^2\frac{1-r^2}{|1-rz|^2}}$. The minimum of 
$$
\psi_r^2\frac{1-r^2}{|1-rz|^2}\sim \frac{(1+r^2-2r\cos(t))}{(1-r\cos(t))^2} (1-r^2)
$$
for $t\in[-\pi,\pi]$ occurs at $t=\pm\pi$, and equals $1-r^2$. It follows that
$$
\sigma((P_{N(C_r-I)}-P_{N(C_r+I)})^2)=[1-r^2,1]
$$
with no eigenvalues. The result follows by the spectral symmetry of $P_{N(C_r-I)}-P_{N(C_r+I)}$. Note that $\pm 1$ cannot be eigenvalues of $P_{N(C_r-I)}-P_{N(C_r+I)}$, because in either case $1$ would be an eigenvalue of $(P_{N(C_r-I)}-P_{N(C_r+I)})^2$. 
\end{proof}
Summarizing: $N(C_a-I)$ and $N(C_a+I)$ are in generic position.

When one studies the geometry of a pair of subspaces $\s$, $\t$, an important feature is the spectral picture of the product $P_\s P_\t P_\s$. 
\begin{lem}\label{lema 54}
Suppose that $\s$ and $\t$ are closed subspaces of $\h$. If $P_\s P_\t P_\s$ has a an eigenvalue $\lambda\ne 0, 1$, then $\pm(1-\lambda^2)^{1/2}$ are eigenvalues of $P_\s-P_\t$.
\end{lem}
\begin{proof}
Suppose $P_\s P_\t P_\s f=\lambda f$ for $\lambda\ne 0,1$ and $\|f\|=1$. Then $f\in \s$ and therefore the subspace $\v$ generated  by $f$ and $P_\t f$  is invariant  for $P_\s$ and $P_\t$:
$$
P_\s f=f \ ; \ P_\s P_\t f=P_\s P_\t P_\s f=\lambda f \ ; \ P_\t f  \ \hbox{ and } \ P_\t P_\t f=P_\t f
$$ 
belong to $\v$. Then $P_\s-P_\t$ is a selfadjoint operator acting in $\v$, which is two dimensional (if $P_\t f$ where a multiple of $f$, then either $P_\t f=0$ and then $P_\s P_\t P_\s f=P_\s P_\t f=0$; or $P_\t f =\alpha f$ and thus $f\in \s\cap\t$ and therefore  $P_\s P_\t P_\s f=f$). 
Then $\{f, P_\t f\}$ is a basis for $\v$ and the matrices of $P_\s$, $P_\t$ and $P_\s-P_\t$ as operators in $\v$ for this basis are, respectively:
$$
\left( \begin{array}{cc} 1 & \lambda \\ 0 & 0 \end{array} \right) , \ \left( \begin{array}{cc} 0 & 0 \\ 1 & 1 \end{array} \right) \ \hbox{ and } \left( \begin{array}{cc} 1 & \lambda \\ -1 & -1 \end{array} \right).
$$
Note that $\lambda\in(0,1)$: $P_\s P_\t P_\s\ge 0$. The eigenvalues of the third matrix are $-(1-\lambda^2)^{1/2}$ and $(1-\lambda^2)^{1/2}$.
 \end{proof}

\begin{coro}
Let $a\in\mathbb{D}$, $a\ne 0$. Then $P_{N(C_a-I)}P_{N(C_a+I)}P_{N(C_a-I)}$ has no non zero eigenvalues.
\end{coro}
\begin{proof}
The fact that  $N(C_a-I)\cap N(C_a+I)=\{0\}$, implies that  $1$ is not an  eigenvalue of $P_{N(C_a-I)}P_{N(C_a+I)}P_{N(C_a-I)}$. Thus Lemma \ref{lema 54} applies:
$P_{N(C_a-I)}-P_{N(C_a+I)}$ has no eigenvalues.
\end{proof}

There are several papers dealing with norm, invertibility and geometry of pairs of projections (including sums, differences and products of projections). See for instance the papers \cite{deutsch}, \cite{buckholtz} (or the survey paper \cite{bottcherspitkovsky} and references therein). Among these facts, let us state the following proposition. 
\begin{prop} {\rm (\cite{deutsch}, \cite{chenpang})}

Let $P$,$Q$ be orthogonal projections. Then $\|PQ\|<1$ if and only if 
$P^\perp+Q^\perp$ is invertible.
\end{prop}
In our case, we obtain the preliminary estimation:
\begin{coro}\label{corito}
Let $a\in\mathbb{D}$, $a\ne 0$. Then
$$
\|P_{N(C_a-I)}P_{N(C_a+I)}\|<1.
$$
\end{coro}
\begin{proof}
Let us check that $P_{N(C_a-I)}^\perp+P_{N(C_a+I)}^\perp$ is invertible. 
Note that
$$
P_{N(C_a-I)}^\perp+P_{N(C_a+I)}^\perp=2-P_{N(C_a-I)}+P_{N(C_a+I)}=2- (C_a+I)M_{\psi_r}-(I-C_a)M_{\psi_a}=2-2M_{\psi_a},
$$
which is invertible if and only if the continuous function $\psi_a-1$ does not vanish in $\mathbb{T}$. 
Note that
$$
\psi_a(z)-1=\psi_a(z) \frac{1-|a|^2}{|1-\bar{a}z|},
$$
does not vanish in $\mathbb{T}$.
\end{proof}

We may refine this result, by yet another consequence of Lemma \ref{lema 54}.
\begin{coro}\label{coro 58}
Let $a\in\mathbb{D}$, $a\ne 0$. Then 
$$
\sigma(P_{N(C_a-I)}P_{N(C_a+I)}P_{N(C_a-I)})\subset[0, |a|].
$$
with $|a|$ belonging to this spectrum. In particular, $\|P_{N(C_a-I)}P_{N(C_a+I)}P_{N(C_a-I)}\|= |a|$. 
\end{coro}
\begin{proof}
Write $A:=P_{N(C_a-I)}P_{N(C_a+I)}P_{N(C_a-I)}$. Note that $A\ge 0$. Consider the universal representation  $\pi:\b(\h)\to \b(\j)$, obtained as the sum of all GNS representations of $\b(\h)$. 
Note that $\|A\|$ is an eigenvalue for $\pi(A)$ (if $\varphi$ is a state in $\b(\h)$ such that $\varphi(A)=\|A\|$, then, in the Hilbert space $\j_\varphi$ associated to GNS representation corresponding to the state $\varphi$, $\|A\|$ is an eigenvalue $\pi(A)$, with associated eigenvector in $\j_\varphi$, given by the class of $1$ in the space $\j_\varphi\subset\j$).
Note from Corollary \ref{corito}, that $\|A\|<1$. Then since $\|A\|\ne 0,1$ is an eignevalue of $\pi(A)=\pi(P_{N(C_a-I)}P_{N(C_a+I)}P_{N(C_a-I)})$, by Lemma \ref{lema 54}, $(1-\|A\|^2)^{1/2}$ is an eigenvalue of
$\pi(P_{N(C_a-I)}-P_{N(C_a+I)})$. Since $\pi$ is faithful, it is an injective $*$-homomorphism, and thus
$$
\sigma(\pi(P_{N(C_a-I)}-P_{N(C_a+I)}))=\sigma(P_{N(C_a-I)}-P_{N(C_a+I)}).
$$ 
i.e., (using Theorem \ref{coro davis})
$$  
(1-\|A\|^2)^{1/2}= (1-|a|^2)^{1/2},
$$
i.e., $\|A\|= |a|$. 
\end{proof}

Note that if $a=0$, the product $P_{N(C_0-I)}P_{N(C_0+I)}P_{N(C_0-I)}=P_\e P_\o P_\e=0$ and this  result is trivially valid.

\section{Pushing with the $R$ and $W$ symmetries}

Recall from (\ref{boludez}) that 
$$
R_aC_aR_a=C_a^*=M_{\frac{1-|a|^2}{|1-\bar{a}z|^2}}C_a.
$$
If we conjugate $C_a$ with the symmetry $W_a$ of Remark \ref{otras simetrias}.2, we get:
\begin{lem}\label{WaCaWa}
Let $a\in\mathbb{D}$. Then
$$
W_aC_aW_a=M_{\frac{1-|a|^2}{(1-\bar{a}z)^2}}C_a=M_{\frac{k_a}{\|k_a\|_2}}W_a.
$$
\end{lem}
\begin{proof}
Direct computation:
$$
W_aC_aW_af(z)=W_aC_a\left(\frac{(1-|a|^2)^{1/2}}{1-\bar{a}z}f(\varphi_a(z))\right)=W_a\left(\frac{1-\bar{a}z}{(1-|a|^2)^{1/2}}f(z)\right)
$$
$$
=\frac{1-|a|^2}{(1-\bar{a}z)^2}f(\varphi_a(z)).
$$
\end{proof}

Recall the definition of $\Omega_a$ in (\ref{pre punto fijo}).
We have the following:
\begin{lem}\label{conjugacion Ra}
Let $a\in\mathbb{D}$. Then
\begin{enumerate}
\item
$$
R_a C_0 R_a=R_{\Omega_a}.
$$
\item
$$
W_a C_0 W_a=W_{\Omega_a}
$$
\end{enumerate}
\end{lem}
\begin{proof}
We prove the first assertion, the second is analogous:
$$
R_aC_0R_af(z)=R_aC_0\left(\frac{(1-|a|^2)^{1/2}}{|1-\bar{a}z|}f(\frac{a-z}{1-\bar{a}z})\right)=R_a\left(\frac{(1-|a|^2)^{1/2}}{|1+\bar{a}z|}f(\frac{a+z}{1+\bar{a}z})\right)
$$
$$
=\frac{(1-|a|^2)}{|1-\bar{a}z|}\frac{1}{|1+\bar{a}\frac{a-z}{1-\bar{a}z}|}         f\left(\frac{a+\frac{a-z}{1-\bar{a}z}}{1+\bar{a}\frac{a-z}{1-\bar{a}z}})\right)=\frac{1-|a|^2}{|1+|a|^2-2\bar{a}z|} f\left(\frac{a-|a|^2z+a-z}{1-\bar{a}z+|a|^2-\bar{a}z}\right)
$$
$$
=\frac{1-|a|^2}{1+|a|^2}\frac{1}{|1-\bar{\Omega}_az|}f\left(\frac{\frac{2a}{1+|a|^2}-z}{1-\frac{2\bar{a}}{1+|a|^2}z}\right)= \frac{1-|a|^2}{1+|a|^2}\frac{1}{|1-\bar{\Omega}_az|}f(\varphi_{\Omega_a}(z)).
$$
Finally, note that
$$
1-|\Omega_a|^2=1-\frac{4|a|^2}{(1+|a|^2)^2}=\frac{(1-|a|^2)^2}{(1+|a|^2)^2},
$$
so that the computation above equals
$$
(1-|\Omega_a|^2)^{1/2}\frac{1}{|1-\bar{\Omega}_az|}f(\varphi_{\Omega_a}(z))=R_{\Omega_a}f(z).
$$
\end{proof}
We can use the formulas in Remark \ref{boludez} and Lemma \ref{conjugacion Ra}.1 applied to the intersections computed in Theorems \ref{44},  \ref{posicion 1} and  \ref{posicion 2} to obtain the following:
\begin{teo}
Let $a\in\mathbb{D}$, Then we have:
\begin{enumerate}
\item
$N(R_{\Omega_a}-I)\cap N(C_a+I)^\perp=\langle \frac{1}{|1-\bar{a}z|}\rangle$.
\item
$N(R_{\Omega_a}+I)\cap N(C_a-I)^\perp=\{0\}$.
\item
$N(R_{\Omega_a}-I)\cap N(C_a+I)=\{0\}$.
\item
$N(R_{\Omega_a}+I)\cap N(C_a+I)^\perp=\{0\}$.
\item
$N(R_{\Omega_a}+I)\cap N(C_a-I)=\{0\}$.
\item
$N(R_{\Omega_a}+I)\cap N(C_a-I)^\perp=\{0\}$.
\item
$N(R_{\Omega_a}-I)\cap N(C_a-I)=\{0\}$.
 \item
$N(R_{\Omega_a}+I)\cap N(C_a+I)=\{0\}$.
\end{enumerate}
\end{teo}
\begin{proof}
We write the proofs of the first two, and indicate how to do the rest.

1. From Theorem \ref{44} we have that $N(C_0-I)\cap N(C_a-I)=\langle 1 \rangle$. Apply the symmetry $R_a$ to this formula, noting that $R_a(N(C_b\pm I))=N(R_aC_bR_a\pm I)$:
$$
\langle R_a1 \rangle =R_a(N(C_0-I))\cap R_a(N(C_a-I))=N(R_aC_0R_a-I)\cap N(R_aC_aR_a-I)
$$
$$
=N(R_{\Omega_a}-I)\cap N(C_a^*-I).
$$
Note that $R_a1=\frac{(1-|a|^2)^{1/2}}{|1-\bar{a}z|}C_a1=\frac{(1-|a|^2)^{1/2}}{|1-\bar{a}z|}$, and that $R(C_a\pm I)=N(C_a\mp I)$, so that
$$
N(C_a^*-I)=R(C_a-I)^\perp=N(C_a+I)^\perp.
$$

2. From Theorem \ref{44} we have that $N(C_0+I)\cap N(C_a+I)=\{0\}$. Applying $R_a$ as above  we get
$$
\{0\}=R_a(N(C_0+I))\cap R_a(N(C_a+I))=N(R_{\Omega_a}+I)\cap N(C_a^*+I)=N(R_{\Omega_a}+I)\cap R(C_a+I)^\perp
$$
$$
=N(R_{\Omega_a}+I)\cap N(C_a-I)^\perp.
$$ 

3. Apply $R_a$ to $N(C_0-I)\cap N(C_a-I)^\perp=\{0\}$ from Theorem \ref{posicion 1}.

4. Apply $R_a$ to $N(C_0+I)\cap N(C_a-I)=\{0\}$ from Theorem \ref{posicion 1}.

5. Apply $R_a$ to $N(C_0+I)\cap N(C_a+I)^\perp=\{0\}$ from Theorem \ref{posicion 1}.

6. Apply $R_a$ to $N(C_0+I)\cap N(C_a-I)=\{0\}$ from Theorem \ref{posicion 1}.

7. Apply $R_a$ to $N(C_0-I)\cap N(C_a+I)^\perp=\{0\}$ from Theorem \ref{posicion 2}.

8. Apply $R_a$ to $N(C_0+I)\cap N(C_a-I)^\perp=\{0\}$ from Theorem \ref{posicion 2}.
\end{proof}

\begin{lem}
Let $a\in\mathbb{D}$, then
\begin{enumerate}
\item
$$
R_{\omega_a}C_aR_{\omega_a}= M_{\displaystyle{\frac{|1+\bar{\omega}_az|}{|1-\bar{\omega}_az|}}} C_0.
$$
\item
$$
W_{\omega_a}C_aW_{\omega_a}= M_{\displaystyle{\frac{1+\bar{\omega}_az}{1-\bar{\omega}_az}}} C_0.
$$

\end{enumerate}
\end{lem}
\begin{proof}
We prove only assertion 1., the proof of assertion 2. is similar:
$$
R_{\omega_a}C_aR_{\omega_a}f(z)=R_{\omega_a}C_a\left(\frac{(1-|\omega_a|^2)^{1/2}}{|1-\bar{\omega}_az|} f(\varphi_{\omega_a}(z))\right)=R_{\omega_a}\left(\frac{(1-|\omega_a|^2)^{1/2}}{|1-\bar{\omega}_a\frac{a-z}{1-\bar{a}z}|} f(\varphi_{\omega_a}(\varphi_a(z)))\right).
$$
Recall from (\ref{phis 1}) that $\varphi_{\omega_a}\circ \varphi_a=-\varphi_{\omega_a}$. Then the above expression equals
$$
R_{\omega_a}\left(\frac{(1-|\omega_a|^2)^{1/2}|1-\bar{a}z|}{|1-\bar{\omega}_a-z(\bar{a}-\bar{\omega}_a)|}f(-\varphi_a(z))\right)=\frac{(1-|\omega_a|^2)|1-\bar{a}\frac{\omega_a-z}{1-\bar{\omega}_az}|}{|1-\bar{\omega}_az||1-\bar{\omega}_aa-(\bar{\omega}_a-\bar{a})\frac{\omega_a-z}{1-\bar{\omega}_az}|}f(-z)
$$
$$
=\frac{(1-|\omega_a|^2)|1-\bar{a}\omega_a -z(\bar{\omega}_a-\bar{a})|}{|1-\bar{\omega}_az||1-\bar{\omega}_aa-\omega_a\bar{a}+|\omega_a|^2-z(2\bar{\omega}_a-\bar{a}-\bar{\omega}_a^2a)|}f(-z).
$$
Note that $2\bar{\omega}_a-\bar{a}-\bar{\omega}_a^2a=0$. Also note that
$$
1-\bar{a}\omega_a -z(\bar{\omega}_a-\bar{a})=(1-\bar{a}\omega_a)(1-z\frac{\bar{\omega}_a-\bar{a}}{1-\bar{a}\omega_a})
$$
and that $\frac{\bar{\omega}_a-\bar{a}}{1-\bar{a}\omega_a}=\overline{\frac{\omega_a-a}{1-\bar{\omega}_aa}}= \overline{\varphi_{\omega_a}(a)}=-\bar{\omega}_a$, because (\ref{phis 1}) implies that
$$
-\omega_a=-\varphi_{\omega_a}(0)=\varphi_{\omega_a}(\varphi_a(0))=\varphi_{\omega_a}(a).
$$
Thus we get
$$
R_{\omega_a}C_aR_{\omega_a}f(z)=\frac{(1-|\omega_a|^2)|1-\bar{a}\omega|}{|1+|\omega_a|^2-\bar{\omega}_aa-\omega_a\bar{a}|} \frac{|1+\bar{\omega}_az|}{|1-\bar{\omega}_az|} f(-z).
$$
The proof finishes by showing that
$$
\frac{(1-|\omega_a|^2)|1-\bar{a}\omega|}{|1+|\omega_a|^2-\bar{\omega}_aa-\omega_a\bar{a}|}=1,
$$
which is a straightforwad computation.
\end{proof}
\begin{rem}
Note that $M_{\frac{|1+\bar{\omega}_az|}{|1-\bar{\omega}_az|}}$ is positive, and that
$$
R_{\omega_a}C_a^*R_{\omega_a}=C_0M_{\frac{|1+\bar{\omega}_az|}{|1-\bar{\omega}_az|}}.
$$
So that this expression is the (unique) polar decomposition of $R_{\omega_a}C_a^*R_{\omega_a}$. In particular
$$
|R_{\omega_a}C_a^*R_{\omega_a}|=R_{\omega_a}|C_a^*|R_{\omega_a}=M_{\frac{|1+\bar{\omega}_az|}{|1-\bar{\omega}_az|}}.
$$
\end{rem}

\section{Geometry of  the Grassmann manifold of $\h$}

Let $Gr(\h):=\{\s\subset\h: \s \hbox{ is a closed subspace of } \h\}$. In \cite{pr} and \cite{cpr proyecciones}, H. Porta, L. Recht and G. Corach (in the latter paper) studied the differential geometric structure of $Gr(\h)$, by identifying a subspace $\s$ with the orthogonal projection $P_\s$  onto $\s$, or alernatively, with the symmetry $\epsilon_\s:=P_\s-I$ which is the identity at $\s$ (see also \cite{grassmann} for an abridged survey of these results).

Let us briefly recall the facts of the geometry of $Gr(\h)$ needed here:
\begin{rem}\label{datos grass} \cite{pr}, \cite{cpr proyecciones}

\noindent

\begin{enumerate}

\item
$Gr(\h)$, regarded as the space of orthogonal projections, is a complemented submanifold of $\b(\h)$. Its tangent spaces are  complemented subspaces of $\b(\h)$. Thus $Gr(\h)$ has the  Finsler metric which consists in the norm of $\b(\h)$ at every tangent space.
\item
Additionally $Gr(\h)$ is a homogeneous space of the unitary group $\u(\h)$ of $\h$ (as in the classical finite dimensional setting): the natural left action of (unitary) operators on closed subspaces, $U\cdot\s=U(\s)$, for $U\in\u(\h)$ and $\s\in Gr(\h)$ (which at the operator level is $U\cdot P_\s=P_{U(\s)}=UP_\s U^*$, or $U\cdot \epsilon_\s=U\epsilon_\s U^*$).
\item
This action admits a natural {\it reductive} structure, based on the notion of diagonaltiy / co-diagonality at every point $\s\in Gr(\h)$: an operator $A\in\b(\h)$ is diagonal with respect to $\s$ if $A\s\subset \s$ and $A\s^\perp\subset\s^\perp$ (i.e., $A$ commutes with $\epsilon_\s$); $B\in\b(\h)$ si co-diagonal with respect to $\s$ is $B\s\subset\s^\perp$ and $B\s^\perp\subset\s$ (this is equivalent to saying that $B$ anti-commutes with $\epsilon_\s$). Clearly any operator in $\b(\h)$ decomposes uniquely as an $\s$-diagonal operator plus an $\s$-co-diagonal operator. The tangent space of $Gr(\h)$ at $\s$ identifies with the space of selfadjoint operators in $\b(\h)$ which are $\s$-codiagonal.  The reductive structure induces a linear connection: if $X(t)$ is a vector field which is tangent along the smooth curve $\s_t$ in $Gr(\h)$ for $t\in[a,b]$ (i.e., $X(t)$ is a smooth path of selfadjoint operators, pointwise $\s_t$-co-diagonal at every $t\in[a,b]$), then the covariant derivative of $X(t)$ is
$$
\frac{DX(t)}{dt}=\s_t-\hbox{co-diagonal part of } \frac{d}{dt}X(t),
$$
where $\frac{d}{dt}X(t)$ is the usual derivative of $X(t)$.
\item
The geodesics of this connection which start at $\s$ are of the form
$$
\gamma(t)= e^{tZ}\s,
$$
where $Z^*=-Z$ is $\s$-codiagonal. We shall say that a geodesic is normalized if additionally $\|Z\|\le\pi/2$. In terms of projections, the geodesic is $\gamma(t)=e^{tZ}P_\s e^{-tZ}$, in terms of symmetries it is $\gamma(t)=e^{tZ}\epsilon_\s e^{-tZ}=e^{2tZ}\epsilon_\s$ (because $Z$ anti-commutes with $\epsilon_\s$)
\item
This item was shown in \cite{p-q} (see also \cite{grassmann}). There exists a geodesic which joins $\s$ and $\t$ if and only if 
$$
\dim \s\cap\t^\perp=\dim \s^\perp\cap\t.
$$
The geodesic can be chosen normalized. There existis a unique normalized geodesic joining $\s$ and $\t$ at $t=0$ and $t=1$ if and only if $\s\cap\t^\perp=\{0\}=\s^\perp\cap\t$.
\item 
Normalized geodesics have minimal length for $|t|\le 1$.
\end{enumerate}

\end{rem}

First, we use the facts contained in Section \ref{subseccion 41} about the (generic) position of $N(C_a-I)$ and $N(C_a+I)$. Note that for $a=0$, $N(C_0-I)=\e$ and $N(C_0+I)=\o=\e^\perp$, we have that $\e\cap\o^\perp=\e$ and $\e^\perp\cap\o=\o$ are both infinite dimensional, and therefore there exist infinitely many geodesics joining $\e$ and $\o$. 
\begin{coro}\label{coro 62}
Let $a\in\mathbb{D}$, $a\ne 0$. Then there exists a unique normalized geodesic joining $N(C_a-I)$ and $N(C_a+I)$ a $t=0$ and $t=1$.
\end{coro}
\begin{proof}
By Theorem \ref{coro davis}, $N(C_a-I)\cap N(C_a+I)^\perp=\{0\}=N(C_a-I)^\perp\cap N(C_a+I)$.
\end{proof}

Next, using Theorems \ref{posicion 1} and \ref{posicion 2} we have the following consequences:
 
\begin{coro}\label{geodesicas 1}
Let $a\ne 0$ in $\mathbb{D}$.
The following pairs subspaces can be joined by a unique normalized geodesic of the Grassmann manifold of $\h=L^2(\mathbb{T})$:
\begin{enumerate}
\item
$N(C_a-I)$ and $\e$.
\item
$N(C_a+I)$ and $\o$.
\item
$N(C_a+I)$ and $\e$.
\end{enumerate}
\end{coro}
We can deduce also the following negative result:
\begin{coro}\label{geodesicas 2}
Let $a\ne 0$. Then the $N(C_a-I)$ and $\o$ cannot be joined by a geodesic of the Grassmann manifold of $\h$.
\end{coro}
\begin{proof}
Note from Theorem \ref{44} (with $b=0$) that 
$N(C_a-I)\cap\o^\perp=N(C_a-I)\cap\e$ is one dimensional. Whereas from Proposition \ref{posicion 2} we have that
$$
N(C_a-I)^\perp\cap\o=\{0\}.
$$
\end{proof}

\begin{rem}\label{bobo}
If $\gamma$ is a geodesic in $Gr(\h)$ and $U\in\u(\h)$, then $U\cdot\gamma$ is also a geodesic in $Gr(\h)$. Therefore $\s$ and $\t$ can be joined by a geodesic (respectively, unique normalized geodesic) if and only if $U\cdot\s$ and $U\cdot\t$ can be joined by a geodesic (respectively, unique normalized geodesic).

The analogous statement is valid for the orthocomplements $\s^\perp$, $\t^\perp$ (the same uniparameter unitary group that induces a geodesic between $\s$ and $\t$, induces a  geodesic between $\s^\perp$ and $\t^\perp$). 
\end{rem} 

We may use these fact to obtain the following consequences:
\begin{coro}\label{si}
Let $a\in\mathbb{D}$, $a\ne 0$. Then the following pairs of subspaces can be joined by a unique normalized geodesic of $Gr(\h)$ at $t=0$ and $t=1$:
\begin{enumerate}
\item $N(C_a+I)$ and $N(R_{\Omega_a}-I)^\perp=N(R_{\Omega_a}+I)$.
\item $N(C_a-I)$ and $N(R_{\Omega_a}+I)^\perp=N(R_{\Omega_a}-I)$.
\item $N(C_a-I)$ and $N(R_{\Omega_a}-I)^\perp=N(R_{\Omega_a}+I)$
\end{enumerate}
\end{coro}
\begin{proof}
By Corollary \ref{geodesicas 1}.1, we know that $N(C_a-I)$ and $N(C_0-I)$ can be joined by a unique normalized geodesic at $t=0$ and $t=1$. Therefore the same holds for 
$$
R_a(N(C_a-I))=N(C_a+I)^\perp \ \hbox{ and } \ R_a(N(C_0-I))=N(R_{\Omega_a}-I),
$$
and therefore also for
$$
N(C_a+I) \ \hbox{ and } \ N(R_{\Omega_a}-I)^\perp.
$$
Assertions 2. and 3. follow similarly using Corollaries \ref{geodesicas 1}.2 and  \ref{geodesicas 1}.3.
\end{proof}
Similarly, using Corollary \ref{geodesicas 2}
\begin{coro}\label{no}
Let $a\in\mathbb{D}$, $a\ne 0$. Then the subspaces
$$
N(C_a+I) \ \hbox{ and } \ N(R_{\Omega_a}+I)^\perp=N(R_{\Omega_a}-I)
$$
cannot be joined by a geodesic of $Gr(\h)$.
\end{coro}
Analogously as with Corollaries \ref{si} and \ref{no}, but using $U=W_a$ we obtain:
\begin{coro}\label{si bis}
Let $a\in\mathbb{D}$, $a\ne 0$. Then the following subspaces can be joined by a unique normalized geodesic of $Gr(\h)$  at $t=0$ and $t=1$:
\begin{enumerate}
\item
$M_{\frac{1}{1-\bar{a}z}}N(C_a-I)$ and $N(W_{\Omega_a}-I)$.
\item
$M_{\frac{1}{1-\bar{a}z}}N(C_a+I)$ and $N(W_{\Omega_a}+I)$.
\item
$M_{\frac{1}{1-\bar{a}z}}N(C_a-I)$ and $N(W_{\Omega_a}+I)$.
\end{enumerate}
\end{coro}
\begin{proof}
Recall from Corollary \ref{geodesicas 1} that there exists a unique geodesic between $N(C_a-I)$ and $\e=N(C_0-I)$, and from Lemmas \ref{WaCaWa} and \ref{conjugacion Ra}.2 the formulas
$$
W_aC_aW_a=M_{\frac{1-|a|^2}{(1-\bar{a}z)^2}}C_a \ \hbox{ and } \ W_aC_0W_a=W_{\Omega_a}.
$$
Note also (after a straightforward computation) that $M_{\frac{\sqrt{1-|a|^2}}{1-\bar{a}z}}C_a=C_a M_{\frac{1-\bar{a}z}{\sqrt{1-|a|^2}}}$, so that
$$
W_aC_aW_a=M_{\frac{\sqrt{1-|a|^2}}{1-\bar{a}z}}C_a M_{\frac{1-\bar{a}z}{\sqrt{1-|a|^22}}}=M_{\frac{\sqrt{1-|a|^2}}{1-\bar{a}z}}C_aM_{\frac{\sqrt{1-|a|^2}}{1-\bar{a}z}}^{-1}.
$$
Then 
$$
W_a(N(C_a-I))=N(W_aC_aW_a-I)=N(M_{\frac{\sqrt{1-|a|^2}}{1-\bar{a}z}}C_aM_{\frac{\sqrt{1-|a|^2}}{1-\bar{a}z}}^{-1}-1)=M_{\frac{\sqrt{1-|a|^2}}{1-\bar{a}z}} N(C_a-I),
$$
and
$$
W_aN(C_0-I)=N(W_aC_0W_a-I)=N(W_{\Omega_a}-I),
$$
and the proof of the first assertion follows from Remark \ref{bobo}. The other two assertions follow similarly. 
\end{proof}
Likewise, from Corollary \ref{geodesicas 2} and similar reasoning we have that
\begin{coro}\label{no bis}
The subspaces
$$
M_{\frac{1}{1-\bar{a}z}}N(C_a-I) \ \hbox{ and } \ N(W_{\Omega_a}+I)
$$
cannot be joined by a geodesic of $Gr(\h)$.
\end{coro}

\begin{rem}
The result of existence of unique normalized geodesics between subspaces, could also be stated in terms of the corresponding projections, or symmetries.
For instance, the fact that there exists unique such geodesics joining
\begin{enumerate}
\item 
the subspaces $N(C_a-I)$ and $N(C_a+I)$ (from Corollary \ref{coro 62}), or
\item
the subspaces $N(C_a+I)$ and $N(R_{\Omega_a}-I)$ (from Corollary \ref{si}.3),
\end{enumerate}
can be rephrased: there exist unique geodesics between (respectively)
\begin{enumerate}
\item
the projections $(I+C_a)M_{\psi_a}$ and $(I-C_a)M_{\psi_a}$ (where $\psi_a(z)=\left(1+\frac{1-|a|^2}{|1-\bar{a}z|^2} \right)^{-1}$),
\item
the symmetries $\epsilon_{N(C_a+I)}$ and $R_{\Omega_a}$.
\end{enumerate}
Note that 
$$
\epsilon_{N(C_a+I)}=2P_{N(C_a+I)}-I=M_{2\psi_a-1}-C_aM_{\psi_a}.
$$
\end{rem}

\end{document}